\documentclass[11pt]{amsart}

\usepackage{bm}
\usepackage{fullpage}
\usepackage{amssymb}
\usepackage{amsmath,amsfonts,amsthm}
\usepackage{hyperref}
\usepackage{color}
\usepackage{cite}

\newtheorem{theorem}{Theorem}

\newtheorem{lemma}[theorem]{Lemma}

\newtheorem{Claim}[theorem]{Claim}

\title{ Multipartite Ramsey number $m_j(K_m, nK_2)$ }

\author{Yaser Rowshan$^1$}
\address{$^1$Y. Rowshan, 
	Department of Mathematics, Institute for Advanced Studies in Basic Sciences (IASBS), Zanjan 45137-66731, Iran}
\email{y.rowshan@iasbs.ac.ir}

\begin{document}
	\maketitle 
	
\begin{abstract}
	Assume that $K_{j\times n}$ be a complete, multipartite graph consisting of $j$ partite sets and $n$ vertices
	in each partite set. For given graphs $G_1, G_2,\ldots, G_n$, the  multipartite Ramsey number (M-R-number) $m_j(G_1, G_2, \ldots,G_n)$ is the smallest integer $t$ such that  for any $n$-edge-coloring $(G^1,G^2,\ldots, G^n)$ of  the edges of $K_{j\times t}$, $G^i$ contains a monochromatic copy of $G_i$  for at least one $i$. The size of  M-R-number $m_j(nK_2, C_m)$  for $j, n\geq 2$ and $4\leq m\leq 6$,
	 the size of  M-R-number $m_j(nK_2, C_7)$  for $j \geq 2$ and $n\geq 2$,   the size of  M-R-number  $m_j(nK_2,K_3)$, for each $j,n\geq 2$, the size of  M-R-number $m_j(K_3,K_3, n_1K_2,n_2K_,\ldots,n_iK_2)$ for $j \leq 6$ and $i,n_i\geq 1$ and the size of  M-R-number $m_j(K_3,K_3, nK_2)$ for $j \geq 2$ and $n\geq 1$ have been computed in several papers up to now. In this article we obtain the values of  M-R-number  $m_j(K_m, nK_2)$, for each $j,n\geq 2$ and each $m\geq 4$.
\end{abstract}

\section{Introduction}

After Erdös and Rado published the  paper \cite{erdos1956partition} in 1956, Ramsey theory has recived attention of many mathematicians because of its vast range of application in several fields such as  graph theory, number theory, geometry, and logic \cite{rosta2004ramsey}. The classical Ramsey problem states that given numbers $n_1,\ldots,n_k$, any $k$-edge-coloring of any sufficiently large complete graph $K_n$ contains a monochromatic complete subgraph with $n_i$ vertices colored with the $i$th color, and ask for the minimum number $n$ satisfying this property. However, there is no loss to work in any class of graphs and their subgraphs instead of complete graphs. One can refer to \cite{barton2016ramsey, benevides2012multipartite}, and \cite{gyarfas2009multipartite} and it references for further studies.

Assume that $K_{j\times n}$ be a complete and multipartite graph consisting of $j$ partite sets and $n$ vertices
in each partite set. For given graphs $G_1, G_2,\ldots, G_n$, the   M-R-number $m_j(G_1, G_2, \ldots,G_n)$ is the smallest integer $t$, so that  for any $n$-edge-coloring $(G^1,G^2,\ldots, G^n)$ of  the edges of $K_{j\times t}$, $G^i$ contains a monochromatic copy of $G_i$  for at least on $i$.

In \cite{burger2004ramsey}, Burgeet et al. studied the M-R-number $m_j(G_1, G_2)$, where both $G_1$ and $G_2$ are a complete  multipartite graphs. Recently the  M-R-number $m_j(G_1, G_2)$ has been studied for special classes of graphs, see \cite{yek, Anie, burger2004diagonal, gholami2021bipartite} and its references, which can be naturally extent to several colors, see \cite{sy2011size, rowshan2021proof, sy2010size, sy2009path, luczak2018multipartite, lakshmi2020three}. The  M-R-number $m_j(K_{1,m},G)$, for $j = 2,3$ where $G$ is a $P_n$ or a $C_n$ is determined in \cite{lusiani2017size}. In \cite{jayawardene2016size}, aouthers determined the  M-R-number $m_j(nK_2, C_n)$ where $n\in \{3,4,5,6\}$  and $j\geq 2$. 
The values of M-R-number $m_j(nK_2, C_7)$ where $j\leq 4$ and $n\geq2$  determined in \cite{math9070764} and the values of  M-R-number $m_j(nK_2, C_7)$ where $j\geq 5$ and $n\geq 1$ computed in \cite{rowshan2021multipartite}.
The size of  M-R-number $m_j(K_3,K_3, n_1K_2,n_2K_,\ldots,n_iK_2)$ for $j \leq 6$ and $i,n_i\geq 1$ and the size of  M-R-number $m_j(K_3,K_3, nK_2)$ for $j \geq 2$ and $n\geq 1$ have been computed in  \cite{rowshan1}.

In this article we obtain the values of  M-R-number  $m_j(K_4, nK_2)$, for $j\geq 2$ and small $ n=4,5$, also we obtain some  bounds of  M-R-number  $m_j(K_m, nK_2)$, where $j,m\geq 4$ and $n\geq 2$. In particular we prove  the following results:


\begin{theorem}\label{Mth}
	For each positive integer $j\geq m$, $m\geq 4$ and  $n\geq 3$, we have: 
	\[m_j(K_m, nK_2)=   \lceil\frac{2n}{j+2-m}\rceil .\]	
\end{theorem}
All the graphs  $G$ considered in this paper are undirected, simple, and finite graphs.
The order of a graph $G$ is defined by $|V(G)|$, which $V(G)$ denote the set of vertices of $G$. A $n$ stripe of a graph $G$ is defined as a set of $n$ edges without  common vertex.
For a vertex $v\in V(G)$, the degree and neighbours of $v$ are denoted by $\deg_G{(v)}$ and $N_G(v)$ , respectively. The neighbourhood of a vertex $v\in V(G) \cap X_j$ is defined by $N_{X_j}(v)=\lbrace u \in V(X_j) \vert uv \in E(G) \rbrace$. A cycle   on $n$ vertices is denoted by $C_n$. We use $[X_i,X_j]$ to denote the set of edges between partite sets $X_i$ and $X_j$. The complement of a graph $G$ is denoted by ${\overline G}$. The join of two graphs $G$ and $H$, defined by $G\oplus H$, is a graph obtained from $G$ and $H$ by joining each vertex of $G$ to all vertices of $H$.  The union of two graphs $G$ and $H$, define by $G\cup H$, is a graph obtain from $G$ and $H$, where $V(G\cup H) =V(G)\cup V(H)$ and $E(G\cup H) =E(G)\cup E(H)$. For convenience, we use $[n]$  instead of $\{1,2,\ldots,n\}$. $G$ is $n$-colorable to $(G_1, G_2,\ldots, G_n)$ if there exist a $n$-edge decomposition of $G$, say $(G^1, G^2,\ldots, G^n)$ where $G_i\nsubseteq G^i$ for each $i=1,2, \ldots,n.$


\section{Proof of Theorem \ref{Mth}}

\subsection{Some results in term of the M-M-numbers related to stripes versus $K_n$.}
In this section, we prove some results in term of the size  Multipartite Ramsey numbers related to stripes versus $K_n$. We begin with the following theorem:
\begin{theorem}\label{t1}\cite{jayawardene2016size} If $j \geq 2$, then:

	\[m_j(K_3, nK_2)=  \left\lbrace
\begin{array}{ll}
	
\infty& ~~~~~~~~~ ~if ~j=2~~, ~~\vspace{.2 cm}\\
	\lceil\frac{2n}{j-1}\rceil& ~~~~~~ otherwise.~\vspace{.2 cm}\\
	
\end{array}
\right.
\]	
	
\end{theorem}
 As any $m-1$-bipartite  graph is $K_m$-free, we get the size of  M-R-number $m_j(K_m, n_1K_2,\ldots,n_iK_2)$, for each $ j\leq m-1, i, n_i\geq 1$ in the next theorem.
\begin{theorem}\label{t2} For each positive integers  $j,i, n_i$, where $2\leq j\leq m-1$ and $ i, n_i\geq 1$, we have the following:
	\[m_j(K_m, n_1K_2,n_2K_,\ldots,n_iK_2)= \infty.\]
	
\end{theorem}
 
Suppose that $j\geq 4$  be  positive integer, in the next two results  we get the exact value of  M-R-number $m_j(K_j, nK_2)$, for each $j\geq4$ and $n\geq 1$.
\begin{lemma}\label{l1}
	For each positive integer $n\leq 2$:
	\[m_j(K_j, nK_2)=n.\]
\end{lemma}
\begin{proof}
For $n=1$ it is clear that $m_j(K_j, K_2)=1$, hence assume that $n=2$. By considering $K_j=K_3\cup (K_{j-3} \oplus 3K_1)$ it can be check that $m_j(K_j, 2K_2)\geq 2$. As $n=2$ and any part of $K=K_{j\times 2}$ has two vertices, it is easy to say that in any two coloring of the edges of $K$ say $(G^1, G^2)$, either $K_j\subseteq G^1$ or $2K_2\subseteq G^2$. Hence $m_j(K_j, nK_2)=n$	for each $j\geq 2$ and $n=1,2$.
\end{proof}
\begin{theorem}\label{t3}
	For each positive integers  $j\geq 2$ and $ n\geq 3$:
	\[m_j(K_j, nK_2)=n.\]
\end{theorem}

\begin{proof}
For $n=1,2$, the proof is complete by Lemma \ref{l1}, hence assume that $n\geq 3$.  Let $K=K_{j\times(n-1)}$ and consider $2$-edge-coloring $(G^1,G^2)$ of  $K$, where $G^1\cong K_{(j-1)\times (n-1)}$ and   $\overline{G^1}=G^2$. It is easy to say that $G^2$ is a $2$-partite graph with $n-1$ vertices in the first part  and $(j-1)(n-1)$ vertices in the second part. That is  $G^2\cong K_{n-1, (j-1)(n-1)}$. By definition $G^i$, it can be check that $K_j\nsubseteq G^1$ and $nK_2\nsubseteq G^2$, that is $K$ is 2-colorable to $(K_j,nK_2)$, which means that $m_j(K_j, nK_2)\geq n$ for each $j\geq 2$ and each $n\geq 1$. 

 We prove $m_j(K_j, nK_2)\leq n$ by induction on $n$. For $n=1,2$ theorem holds by Lemma \ref{l1}. Assume that for each, $n'\leq n-1$, the theorem holds. By contrary suppose that  $K=K_{j\times n}$  with partite sets $X_i=\{x_1^i,x_2^i,\ldots,x_{n}^i\}$ for each $i\in[j]$ is $2$-colorable to $(K_j,nK_2)$, that is there exist 2-edge-coloring $(G^1,G^2)$ of  $K_{j\times n}$, where $K_j\nsubseteq G^1$ and $nK_2\nsubseteq G^2$. Without loss of generality (w.l.g) assume that $e=x_1^1x_1^2\in E(G^2)$.  Now, set $X'_i=X_i\setminus \{x_1^i\}$ for each $i\in[j]$,  and consider   $K'=K_{j\times (n-1)}=G[X'_1,\ldots,X'_j]$. Consider a 2-edge-coloring $(G'^1,G'^2)$ of  $K'$. As $m_j(K_j, (n-1)K_2)\leq n-1$,   $K_j\nsubseteq G'^1\subseteq G^1$,  one can check that $(n-1)K_2\subseteq G'^2\subseteq G^2$. Suppose that $M'$ be a M-M in $G'^2$. Therefore, $M=M'\cup \{e\}$ be a matching of size $n$ in $G^2$, a contradiction. Which means that $m_j(K_j, nK_2)\leq n$,  for each $n\geq 2$. Therefore $m_j(K_j, nK_2)=n$ for each $j\geq 2$ and each $n\geq 1$, and  the proof is complete.
\end{proof}
Suppose that $j\geq 2$, it is clear that $m_j(K_{j-1}, K_2)=1$. Also by considering $K_j=(K_{j-3}\oplus 3K_1 )\cup K_3$ it can be check that $m_j(K_{j-1}, 2K_2)\geq 2$ and  as $K=K_{j\times 2}$ has two disjoint copy of $K_j$ it is easy to say that in any two coloring of the edges of $K$ say $(G^1, G^2)$, either $K_{j-1}\subseteq G^1$ or $2K_2\subseteq G^2$. Hence $m_j(K_{j-1}, 2K_2)=2$. Now assume that $n\geq 3$, in the following results  we get the exact value of  M-R-number $m_j(K_{j-1}, nK_2)$, for each $j\geq 2$ and each $n\geq 3$. We begin with the following theorem:
\begin{theorem}\label{t4}
	Suppose that  $n\in \{3,4,5\}$, then:
	\[m_j(K_{j-1}, nK_2)=n-1. \]
 
\end{theorem}
\begin{proof}
 Consider a $2$-edge-coloring $(G^1,G^2)$ of  $K=K_{j\times (n-2)}$, where $n\in[j]$, $G^2\cong K_{3\times (n-2)}$ and   $\overline{G^2}=G^1$. It is easy to say that $G^1$ is a $(j-2)$-partite graph with $n-2$ vertices in $(j-3)$ parts  and $3(n-2)$ vertices in one parts. That is  $G^1\cong K_{n-2,\ldots, n-2, 3(n-2)}$. By definition $G^1$, one can check that $K_{j-1}\nsubseteq G^1$, also as $|V(G^2)|=3(n-2)\leq 2n-1$ for each $n\in \{3,4,5\}$, hence it can be say  that  $nK_2\nsubseteq G^2$, that is $K$ is 2-colorable to $(K_{j-1},nK_2)$, which means that $m_j(K_{j-1}, nK_2)\geq n-1$ for each $n\in \{3,4,5\}$. 

Now for each $n\in \{3,4,5\}$, consider  $K=K_{j\times (n-1)}$  with partite sets $X_i=\{x_1^i,x_2^i,\ldots,x_{n-1}^i\}$ for each $i\in[j]$. Suppose that $(G^1,G^2)$ be a 2-edge-coloring of  $K$, where $K_j\nsubseteq G^1$. Assume that $n=3$ and w.l.g let that $e=x_1^1x_2^1\in E(G^2)$.  Now, set $X'_i=X_i\setminus \{x_1^i\}$ for $i=1,2$ and $X'_i=X_i$ for each $i\geq 3$,  and consider   $K'=G[X'_1,\ldots,X'_j]$. It can be check that $K'\cong K_2\oplus K_{(j-2)\times 2}$, therefore $K_{(j-1)\times 2}\subseteq K'$. Consider a 2-edge-coloring $(G'^1,G'^2)$ of  $K'$. As by Theorem \ref{t3} $m_{j-1}(K_{j-1}, 2K_2)= 2$,  $K_{{j-1}\times 2}\subseteq K'$, and $K_{j-1}\nsubseteq G'^1\subseteq G^1$,  one can check that $2K_2\subseteq G'^2\subseteq G^2$. Suppose that $M'$ be a M-M in $G'^2$. Therefore, $M=M'\cup \{e\}$ is a matching of size $3$ in $G^2$, which means that $m_j(K_{j-1}, 3K_2)=2$.  Since by Theorem \ref{t3} $m_j(K_j, nK_2)= n$ for each $n\geq 3$, then for $n=4,5$ by the proof is same as the case that $n=3$, we have $m_j(K_{j-1}, nK_2)=n-1$, so the proof is complete.
\end{proof}
 It is clear that    $m_j(K_4, K_2)=1$ for each $j\geq 4$ also it is clear that $m_j(K_4, 2K_2)=1$ for each $j\geq 6$. In the following results  we get a general lower bounds of  M-R-number $m_j(K_m, nK_2)$, for each $j\geq m-2$ and  each $n\geq 3$.
\begin{theorem}\label{t5}
	Suppose that  $j\geq m+1, m\geq 4$ and $n\geq 3$, then:
	\[m_j(K_{m}, nK_2)\geq \lceil \frac{2n}{j+2-m}\rceil .\]
	
\end{theorem}
\begin{proof}
	Consider a $2$-edge-coloring $(G^1,G^2)$ of  $K=K_{j\times (t-1)}$, where $n\geq 3$, $t=\lceil \frac{2n}{j+2-m}\rceil$, $G^2\cong K_{(j+2-m)\times (t-1)}$ and   $\overline{G^2}=G^1$. It is easy to say that $G^1$ is a $m-1$-partite graph with $t-1$ vertices in $(m-2)$ parts  and $(j+2-m)(t-1)$ vertices in the last parts. That is  $G^1\cong K_{t-1, \ldots,t-1, (j+2-m)(t-1)}$. By definition $G^1$, one can check that $K_m\nsubseteq G^1$, also as $|V(G^2)|=(j+2-m)(t-1)$ and $t=\lceil \frac{2n}{j+2-m}\rceil$, then one can say that $ |V(G^2)|\leq 2n-1$ for each $n\geq 3$ and each $j\geq m+1$, hence $nK_2\nsubseteq G^2$, that is $K$ is 2-colorable to $(K_{m}, nK_2)$, which means that $m_j(K_{m}, nK_2)\geq t$ for each $j\geq m+1, m\geq 4$ and each  $n\geq 3$. 
\end{proof}
\subsection{Size  Multipartite Ramsey numbers related to stripes versus $K_4$}
In this section, we obtain the values of M-R-number  $m_j(K_4, nK_2)$ for each  $j\geq 2$ and $n\geq 1$. By Theorem \ref{t2} we have $m_j(K_4, nK_2)=\infty$ for $j=2,3$, also by Lemma \ref{l1} and Theorem \ref{t3} we have  $m_4(K_4, nK_2)=n$ for  each $n$, and by Theorem \ref{t4} we have  $m_5(K_4, nK_2)=n-1$ for each $n\in \{3,4,5\}$. It can be check that  $m_j(K_4, K_2)=1$ for each $j\geq 4$ also it is clear that $m_j(K_4, 2K_2)=1$ for each $j\geq 6$ and  $m_5(K_4, 2K_2)=2$. Hence assume that $j\geq 5$ and $n\geq 3$. In the following results  we get the size of  M-R-number $m_5(K_4, nK_2)$, for each $n\geq 3$.

\begin{theorem} \label{t6}
	Suppose that  $n\geq 3$, then:
	\[m_5(K_4, nK_2)= \lceil \frac{2n}{3}\rceil .\]
	
\end{theorem}
\begin{proof}
By Theorem \ref{t5} the lower bound holds. So we most show that  $m_5(K_4, nK_2)\leq \lceil \frac{2n}{3}\rceil$. We prove $m_5(K_4, nK_2)\leq \lceil \frac{2n}{3}\rceil$ by induction on $n$. For $n=3,4,5$ theorem holds by Theorem \ref{t4}. Assume that for each, $6\leq n'\leq n-1$, the theorem holds. By contrary suppose that  $K=K_{j\times t}$  with partite sets $X_i=\{x_1^i,x_2^i,\ldots,x_{t}^i\}$ for each $i\in[j]$ and $t=\lceil \frac{2n}{3}\rceil$ is $2$-colorable to $(K_4,nK_2)$, that is there exist 2-edge-coloring $(G^1,G^2)$ of  $K_{5\times t}$, where $K_4\nsubseteq G^1$ and $nK_2\nsubseteq G^2$. For each $n\geq 6$, if $n\neq 3k$, then it can be check that $\lceil \frac{2(n-1)}{3}\rceil+1=\lceil \frac{2n}{3}\rceil$. Therefor assume that $n\neq 3k$ for each $k\geq 2$,  and set $X'_i=X_i\setminus \{x_1^i\}$ for each $i\in[j]$. Now  consider   $K'=K_{j\times (t-1)}=G[X'_1,\ldots,X'_j]$. Consider a 2-edge-coloring $(G'^1,G'^2)$ of  $K'$. As $m_j(K_4, (n-1)K_2)\leq t-1$,   $K_4\nsubseteq G'^1\subseteq G^1$,  one can check that $(n-1)K_2\subseteq G'^2\subseteq G^2$. Suppose that $M'$ be a copy of $(n-1)K_2$ in $G'^2$. Therefore, as $j=5$ and  $K_4\nsubseteq G^1$, one can say that there is at least one edges of $E(G[V'])$ say $e$, so that $e\in E(G^2)$, where $V'=\{x_1^i, ~i=1,2,\ldots,j\}$. So $M=M'\cup \{e\}$ is a matching of size $n$ in $G^2$, a contradiction. Which means that $m_j(K_4, nK_2)\leq t$,  for each $n\geq 3$ where $n\neq 3k$. Now assume that $n=3k$ for some $k\geq 2$. Hence, it can be check that $\lceil \frac{2(n-2)}{3}\rceil+1=\lceil \frac{2n}{3}\rceil=2k$. Set $X'_i=X_i\setminus \{x_1^i\}$ for each $i\in[j]$. Consider   $K'=K_{j\times (t-1)}=G[X'_1,\ldots,X'_j]$ and let  $(G'^1,G'^2)$ be a 2-edge-coloring of  $K'$. As $m_j(K_4, (n-2)K_2)\leq k-1$,   $K_4\nsubseteq G'^1\subseteq G^1$,  one can check that $mK_2\subseteq G'^2\subseteq G^2$ where $m\geq n-2$. Suppose that $M'$ be a M-M  in $G'^2$.  If $|M'|=n-1$,  then, as $j= 5$ and  $K_4\nsubseteq G^1$, one can say that there is at least one edges of $E(G[V'])$ say $e$, so that $e\in E(G^2)$, where $V'=\{x_1^i, ~i=1,2,\ldots,j\}$. So $M=M'\cup \{e\}$ is a matching of size $n$ in $G^2$, a contradiction. Hence assume that $|M'|=n-2$. Now, as $j=5$,  $n=3k$ and  $|M'|=n-2$ it is easy to say that the following claim is true: 
\begin{Claim}\label{c1}For each $(j_1,j_2,j_{3})$, where $j_i\in [5]$, there exist at lest one vertex of $X'_{j_1}\cup X'_{j_2}\cup X'_{j_{3}}$ say $x$, so that $x\notin V(M')$.
	
\end{Claim}
\begin{proof} As, $n=3k$ we have $t=2k$, that is $t$ is even. Also we have $|X'_i|=2k-1=\lceil \frac{2(n-2)}{3}\rceil$. Therefore as $|M'|=n-2$ we have $|V(M')|=2n-4=6k-4$. Hence,  for each $(j_1,j_2,j_{3})$, where $j_i\in [5]$, it can be say that  $|X'_{j_1}|+|X'_{j_2}|+|X'_{j_{3}}|=3(2k-1)=6k-3$. Which means that  for each $(j_1,j_2,j_{3})$ there exist a vertex of $X'_{j_1}\cup X'_{j_2}\cup X'_{j_{3}}$ say $x$, so that $x\notin V(M')$.
\end{proof}
Therefore,   since $j=5$, by Claim \ref{c1}, and by pigeon-hole
principle, there exist  three partition of $K$, say $X_{i_1},  X_{i_2},  X_{i_3}$, so that for each $r\in[3]$ there exist a vertex of $X_{i_r}$ say $v_r$ so that $v_r$ not belongs to $M$. W.l.g assume that $ X_{i_1}=X_i$ and $v_i=x_2^i$ for each $i\in[3]$.  Set $V''=\{x_2^1,x_2^2,x_2^3\}$. Therefore, it can be check that $ K''= G[V'\cup V'']\cong K_2\oplus K_{3\times 2}\subseteq  G^1$.  Therefore $K_{4\times 2}\subseteq K''$. Consider a 2-edge-coloring $(G''^1,G''^2)$ of  $K''$. As by Lemma \ref{l1}, $m_4(K_4, 2K_2)= 2$,  $K_{4\times 2}\subseteq K''$, and $K_4\nsubseteq G''^1\subseteq G^1$,  one can check that $2K_2\subseteq G''^2\subseteq G^2$. Suppose that $M''$ be a M-M in $G''^2$. Therefore, $M=M'\cup M''$ is a matching of size $n$ in $G^2$, which means that $m_5(K_4, nK_2)\leq 2k$ where $n=3k$.
   Therefore for each $n\geq 3$ we have $m_5(K_4, nK_2)\leq \lceil \frac{2n}{3}\rceil$. Which means that for each $n\geq 3$, $m_5(K_4, nK_2)= \lceil \frac{2n}{3}\rceil$ and the proof is complete.
\end{proof}
\begin{theorem} \label{t7}
	Suppose that $j\geq 6$ and   $n\geq 3$. Given that  $m_j(K_4, (n-1)K_2)= \lceil \frac{2n-2}{j-2}\rceil $, then:
	\[m_j(K_4, nK_2)= \lceil \frac{2n}{j-2}\rceil .\]
	
\end{theorem}
\begin{proof}\label{t8}
 To prove $m_j(K_4, nK_2)\leq \lceil \frac{2n}{j-2}\rceil$ consider  $K=K_{j\times t}$  with partite sets $X_i=\{x_1^i,x_2^i,\ldots,x_{t}^i\}$ for each $i\in[j]$ and $t=\lceil \frac{2n}{j-2}\rceil$. Suppose that $K$ is $2$-colorable to $(K_4,nK_2)$, that is there exist 2-edge-coloring $(G^1,G^2)$ of  $K_{j\times t}$, where $K_4\nsubseteq G^1$ and $nK_2\nsubseteq G^2$.  Thus $K_4\nsubseteq G^1[V(K')]$ where $K'=K_{j\times t'}\subseteq K$,  $t'=\lceil \frac{2n-2}{j-2}\rceil$, so it can be check that  $(n-1)K_2\subseteq G^2[V(K')]$.  As $t=\lceil \frac{2n}{j-2}\rceil$ and $|M|=n-1$, where $M$ ba a M-M in $G^2$, then one can show that the following claim is true:
 
 \begin{Claim}\label{c2}
 There exists at least $2t+2$ vertices of $K$, say $Y$ so that the vertices of $Y$ not belongs to $M$.
 	
 \end{Claim}
 \begin{proof} It can be check that $|Y|=|V(K)|-2(n-1)$, hence $|Y|=jt-2(n-1)= 2t +(j-2)t-2(n-1)= 2t +(j-2)\lceil \frac{2n}{j-2}\rceil -2(n-1)$, as $(j-2)\lceil \frac{2n}{j-2}\rceil -2(n-1)\geq 2$, we have 	$|Y|\geq 2t+2$, and the proof of claim is complete.
 \end{proof}
If either $Y\cap X_i\neq \emptyset$ for at least four parts of $K$ or $t=1$, then one can check that $K_4\subseteq G^1[Y]$,  a contradiction. Hence, let $t\geq 2$, as $|Y|=2t+2$ and $|X_i|=t$, we can assume that $Y\cap X_i\neq \emptyset$ for three $i, i\in[j]$. W.l.g let $Y\cap X_i\neq \emptyset$ for each $i\in[3]$. As $t\geq 2$, it can be check that $K''=K_{3\times 2}\subseteq G^1[Y]$. Also as $j\geq 6$ and $|M|=n-1$, it is easy to say that there exist at least one edges of $M$ say $e=v_1v_2$, so that $v_1,v_2\notin X_i$ for $i=1,2,3$. Therefore w.l.g let $v_1\in X_4$ and $v_2\in X_5$. Now, by considering $v_1,v_2$ and $Y$, where $Y$ is the vertices of $K$ not belongs to $M$,  we have the following claim:  
 
 \begin{Claim}\label{c3}
 	W.l.g. let $|N_{G^2}(v_1)\cap Y|\geq |N_{G^2}(v_2)\cap Y| $. If $|N_{G^2}(v_1)\cap Y|\geq 2$, then $|N_{G^2}(v_2)\cap Y|=0$. If $|N_{G^2}(v_1)\cap Y|=1$, then $|N_{G^2}(v_2)\cap Y|\leq 1$ and if $|N_{G^2}(v_1)\cap Y|=1$,  then $v_1$ and $v_2$ have the same neighbor in $Y$.
 \end{Claim}
 \begin{proof} By contradiction. Suppose that $\{v,v'\}\subseteq N_{G^2}(v_1)\cap Y$ and $v''\in N_{G^2}(v_2)\cap Y$, in this case, we set $ M' = (M\setminus \{v_1,v_2\}) \cup\{v_1v,v_2v''\}$. Clearly, $M'$ is a matching with $|M'| =n$, which contradicts the  $nK_2\nsubseteq G^2$. If $|N_{G^2}(v_i)\cap Y|=1$ and $v_i$ has a different neighbor then the proof is  same.
 \end{proof}
Therefore by Claim \ref{c3} it can be check that for at least one $i\geq 4$, there exist at least one vertex of $X_i$, say $w$, so that $|N_{G^2}(w)\cap Y|\leq 1$. Hence $|N_{G^1}(w)\cap Y|\geq |Y|-1$. So, one can check that $K_4\subseteq G^1[V(K'')\cup \{w\}]$, a contradiction again.
So, we have  $m_j(K_4, nK_2)\leq \lceil \frac{2n}{j-2}\rceil $ for each $j\geq 6$ and $n\geq 3$.  Therefore by Theorem \ref{t5}, we have $m_j(K_4, nK_2)= \lceil \frac{2n}{j-2}\rceil $ for each $j\geq 6$ and $n\geq 3$, which means that the proof is complete.
 
\end{proof}
 Combining Theorems  \ref{t2}, \ref{t3}, \ref{t4},  \ref{t6},  and \ref{t7}, and Lemma \ref{l1} we
 obtain the next theorem which characterize the exact value of the M-R-number
 $m_j(K_4, nK_2)$ for each $j\geq 2$ and each $n\geq 1$ as follows:
 \begin{theorem}\label{t1} If $j \geq 2$, and $n\geq 1$, then:

 	\[m_j(K_4, nK_2)=  \left\lbrace
 	\begin{array}{ll}
 		
 		\infty& ~~~~~~~~~ ~if ~j=2,3,~~ ~~\vspace{.2 cm}\\
 		\lceil\frac{2n}{j-2}\rceil& ~~~~~~ otherwise.~\vspace{.2 cm}\\
 		
 	\end{array}
 	\right.
 	\]	
 	
 \end{theorem}
\subsection{Size  Multipartite Ramsey numbers related to stripes versus $K_5$}
In this section, we obtain the values of M-R-number  $m_j(K_5, nK_2)$ for each  $j\geq 2$ and $n\geq 1$. By Theorem \ref{t3} we have $m_j(K_4, nK_2)=\infty$ for $j=2,3,4$, also by Lemma \ref{l1} and Theorem \ref{t3} we have  $m_5(K_5, nK_2)=n$ for  each $n$, and by Theorem \ref{t4} we have  $m_6(K_5, nK_2)=n-1$ for each $n\in \{3,4,5\}$. It can be check that  $m_j(K_5, 2K_2)=1$ for each $j\geq 7$. Now suppose that $j\geq 6$ and $n$ be  positive integers, it is clear that $m_6(K_5, K_2)=1$. Also by considering $K_6=(K_{3}\oplus 3K_1 )\cup K_3$ it can be check that $m_6(K_5, 2K_2)\geq 2$ and  as $K=K_{6\times 2}$ has two disjoint copy of $K_6$ it is easy to say that in any two coloring of the edges of $K$ say $(G^1, G^2)$, either $K_5\subseteq G^1$ or $2K_2\subseteq G^2$. Hence $m_6(K_5, 2K_2)=2$. In the following results  we get the exact value of  M-R-number $m_6(K_5, nK_2)$, for each $ n\geq 6$. We begin with the following theorem:
\begin{theorem}\label{th4}
For each $n\in N= \{6,7,8\}$,	we have:
	\[m_6(K_5, nK_2)=n-2\]
	
\end{theorem}
\begin{proof}
Consider a $2$-edge-coloring $(G^1,G^2)$ of  $K=K_{6\times (n-3)}$, where $n\in N$, $G^2\cong K_{3\times (n-3)}$ and   $\overline{G^2}=G^1$. It is easy to say that $G^1$ is a $4$-partite graph with $n-3$ vertices in $3$ parts  and $3(n-3)$ vertices in the last parts. That is  $G^1\cong K_{n-3,\ldots, n-3, 3(n-3)}$. By definition $G^1$, one can check that $K_{5}\nsubseteq G^1$, also as $|V(G^2)|=3(n-3)\leq 2n-1$ for each $n\in N$, hence $nK_2\nsubseteq G^2$. Which means that $K$ is 2-colorable to $(K_{5},nK_2)$, that is $m_6(K_{5}, nK_2)\geq n-2$ for each $n\in N$. 
	
 Now,let  $K=K_{6\times (n-2)}$  with partite sets $X_i=\{x_1^i,x_2^i,\ldots,x_{n-2}^i\}$. Consider 2-edge-coloring $(G^1,G^2)$ of  $K$, where $K_5\nsubseteq G^1$. Now we consider three cases as follow:
 
 {\bf Case 1}: $n=6$. As $m_6(K_5, 5K_2)=4$ and $K_5\nsubseteq G^1$, we can say that $M'=5K_2\subseteq G^2$. As, $|M'|=5$ and $6K_2\nsubseteq G^2$, then it can be check that the following claim is true:
 \begin{Claim}\label{c4}For each $(j_1,j_2,j_{3})$, where $j_i\in [6]$, there exists two vertex of $X'_{j_1}\cup X'_{j_2}\cup X'_{j_{3}}$ say $x,x'$, so that $x,x'\notin V(M')$.
 \end{Claim}
 Also as $K_5\nsubseteq G^1$, one can say that there exist two parts of $K$ say $X,X'$ so that all vertices of $X$ and $X'$ lie in $V(M)$, where $M$ ba a M-M in $G^2$. W.l.g assume that $X=X_1, X'=X_2$.  Therefore by Claim \ref{c4} for each $i\in \{3,4,5,6\}$ we have $|X_i\cap Y|\geq 2$, where $Y$ be the vertices of $K$ that is not belongs to $M$. W.l.g assume that $X'_i=\{x_1^i, x_2^i\}\subseteq X_i\cap Y$ for each $i\in \{3,4,5,6\}$. Hence, it can say that $K''=K_{4\times 2}\cong G[X'_3,\ldots, X'_6]\subseteq G^1[Y]$. Since for $i=1,2$ $X_i\subseteq V(M)$, $|X_i|=4$ and $|M|=5$, it is easy to check that there exist at least one edges of $E(M)$ say $e$, so that $e\in E(G[X_1X_2])$. Therefore w.l.g let $e=x_1^1x_1^2\in E(M)$. Hence by the proof is same as proof of Claim \ref{c3} it can be check that, there exist at least one vertex of $\{x_1^1,x_1^2\}$, say $w$, so that $|N_{G^2}(w)\cap  V(K'')|\leq 1$. Hence $|N_{G^1}(w)\cap V(K'')|\geq 7$. So, one can check that $K_5\subseteq G^1[V(K'')\cup \{w\}]$, a contradiction again.  Which means that  $m_6(K_5, 6K_2)= 4$. 
 
 {\bf Case 2}: $n=7(8)$. As $m_6(K_5, 6K_2)=4$ and $K_5\nsubseteq G^1$, we can say that $M'=6K_2\subseteq G^2$ for $n=7$( for $n=8$ the proof is same). Therefore the proof is same as the Case 1.
 
  Hence by Cease 1,2 we have the proof is complete.

\end{proof}
\begin{theorem}\label{th5}
	For each $n\geq 6$,	we have:
	\[m_6(K_5, nK_2)=n-\lfloor\frac{n}{3}\rfloor.\]
	
\end{theorem}
\begin{proof}
	Consider a $2$-edge-coloring $(G^1,G^2)$ of  $K=K_{6\times t}$, where $n\geq 6$ and $t=n-\lfloor\frac{n}{3}\rfloor-1$, $G^2\cong K_{3\times t}$ and   $\overline{G^2}=G^1$. It is easy to say that $G^1$ is a $4$-partite graph with $t$ vertices in $3$ parts  and $3t$ vertices in the last parts. That is  $G^1\cong K_{t,\ldots,t, 3t}$. By definition $G^1$, one can check that $K_{5}\nsubseteq G^1$, also as $|V(G^2)|=3t\leq 2n-1$ for each $n\geq 6$, hence one can check that and $nK_2\nsubseteq G^2$, that is $K$ is 2-colorable to $(K_{5}, nK_2)$, which means that $m_6(K_{5}, nK_2)\geq t+1$ for each $n\geq 6$. 
	
	Now,  for each $n\geq 6$, let  $K=K_{6\times t}$  with partite sets $X_i=\{x_1^i,x_2^i,\ldots,x_{t}^i\}$ where $t=n-\lfloor\frac{n}{3}\rfloor$. Consider 2-edge-coloring $(G^1,G^2)$ of  $K$, where $K_5\nsubseteq G^1$.  We prove $m_6(K_5, nK_2)\leq t$ by induction on $n$, for $n=3,\ldots,8$ theorem holds by Theorem \ref{t4} and \ref{th4}. Assume that for each, $9\leq n'\leq n-1$, the theorem holds. Now we consider two cases as follow:
	
	{\bf Case 1}: $\lfloor\frac{n}{3}\rfloor=\lfloor\frac{n-1}{3}\rfloor$. Hence we have $t=n-\lfloor\frac{n}{3}\rfloor=1+(n-1)-\lfloor\frac{n-1}{3}\rfloor=1+t'$, therefore by induction, as  $m_6(K_5, (n-1)K_2)=t'=t-1$ and $K_5\nsubseteq G^1$, we can say that $M'=(n-1)K_2\subseteq G^2[X'_1,\ldots,X'_6]$, where $X_i'=X_i\setminus \{x_1^i\}$ for each $i\in[6]$ and $|X'_i|=t'=t-1$. As, $|M'|=n-1$ and $nK_2\nsubseteq G^2$, then it can be check that $K_6\cong G[\{x_1^1,\ldots,x^6_1\}]\subseteq G^1$, which means that the proof is complete.

	{\bf Case 2}: $\lfloor\frac{n}{3}\rfloor=1+\lfloor\frac{n-1}{3}\rfloor$. Hence we have $t=n-\lfloor\frac{n}{3}\rfloor=n-1-\lfloor\frac{n-1}{3}\rfloor$, therefore by induction, as  $m_6(K_5, (n-1)K_2)=t$ and $K_5\nsubseteq G^1$, we can say that $M'=(n-1)K_2\subseteq G^2$. As, $|M'|=n-1$ and $nK_2\nsubseteq G^2$, then it can be check that the following claim is true:
	\begin{Claim}\label{c5}For each $(j_1,j_2,j_{3})$, where $j_i\in [6]$, there exist at lest two vertex of $X'_{j_1}\cup X'_{j_2}\cup X'_{j_{3}}$ say $x,x'$, so that $x,x'\notin V(M')$.
	\end{Claim}
	Also as $K_5\nsubseteq G^1$, one can say that there exist two parts of $K$ say $X,X'$ so that all vertices of $X$ and $X'$ lie in $V(M)$, where $M$ ba a M-M in $G^2$. W.l.g assume that $X=X_1, X'=X_2$.  Therefore by Claim \ref{c5} for each $i\in \{3,4,5,6\}$ we have $|X_i\cap Y|\geq 2$, where $Y$ be the vertices of $K$ that is not belongs to $M$. W.l.g assume that $X'_i=\{x_1^i, x_2^i\}\subseteq X_i\cap Y$ for each $i\in \{3,4,5,6\}$. Hence, it can say that $K''=K_{4\times 2}\cong G[X'_3,\ldots, X'_6]\subseteq G^1[Y]$. Since for $i=1,2$ $X_i\subseteq V(M)$, $|X_i|=t$ and $|M|=n-1$ and $n\geq9$, it is easy to check that there exist at least one edges of $E(M)$ between $X_1$ and $X_2$. Therefore w.l.g let $x_1^1x_1^2\in E(M)$. Hence by the proof is same as proof of Claim \ref{c3} it can be check that, there exist at least one vertex of $\{x_1^1,x_1^2\}$, say $w$, so that $|N_{G^2}(w)\cap  V(K'')|\leq 1$. Hence $|N_{G^1}(w)\cap V(K'')|\geq 7$. So, one can check that $K_5\subseteq G^1[V(K'')\cup \{w\}]$, a contradiction again.  Which means that  $m_6(K_5, nK_2)= t$. 
	
	Hence by Cease 1,2 we have the proof is complete.

\end{proof}
It can be say that $m_j(K_5, 2K_2)=m_j(K_5, K_2)= 1$ for each $j\geq 7$.  In the following results  we get the exact value of  M-R-number $m_j(K_5, nK_2)$, for each $j\geq 7$ and each $ n\geq 3$.
\begin{theorem} \label{th6}
	Suppose that $j\geq 7$ and   $n\geq 3$. Given that  $m_j(K_5, (n-1)K_2)= \lceil \frac{2n-2}{j-3}\rceil $,	then:
	\[m_j(K_5, nK_2)= \lceil \frac{2n}{j-3}\rceil .\]
	
\end{theorem}
\begin{proof} 
	To prove $m_j(K_5, nK_2)\leq \lceil \frac{2n}{j-3}\rceil$ consider  $K=K_{j\times t}$  with partite sets $X_i=\{x_1^i,x_2^i,\ldots,x_{t}^i\}$ for each $i\in[j]$ and $t=\lceil \frac{2n}{j-3}\rceil$. Suppose that $K$ is $2$-colorable to $(K_5,nK_2)$, that is there exist 2-edge-coloring $(G^1,G^2)$ of  $K_{j\times t}$, where $K_5\nsubseteq G^1$ and $nK_2\nsubseteq G^2$.  Thus $K_5\nsubseteq G^1[V(K')]$ where $K'=K_{j\times t'}\subseteq K$,  $t'=\lceil \frac{2n-2}{j-3}\rceil$, so it can be check that  $(n-1)K_2\subseteq G^2[V(K')]$.  As $t=\lceil \frac{2n}{j-3}\rceil$ and $|M|=n-1$ where $M$ ba a M-M in $G^2$, then one can show that the following claim is true:
	
	\begin{Claim}\label{c6}
		There exists at least $3t+2$ vertices of $K$, say $Y$ so that the vertices of $Y$ not belongs to $M$.
		
	\end{Claim}
	\begin{proof} It can be check that $|Y|=|V(K)|-2(n-1)$, hence $|Y|=jt-2(n-1)= 3t +(j-3)t-2(n-1)= 3t +(j-3)\lceil \frac{2n}{j-3}\rceil -2(n-1)$, as $(j-3)\lceil \frac{2n}{j-3}\rceil -2(n-1)\geq 2$ we have 	$|Y|\geq 3t+2$, and the proof of claim is complete.
	\end{proof}
	If $Y\cap X_i\neq \emptyset$ for at least five parts of $K$, then one can check that $K_5\subseteq G^1[Y]$,  a contradiction. Hence as $|Y|=3t+2$ and $|X_i|=t$, we can assume that $Y\cap X_i\neq \emptyset$ for four $i\in[j]$. W.l.g let $Y\cap X_i\neq \emptyset$ for each $i\in[4]$. As $n\geq 3$, and $nK_2\nsubseteq G^2$ if $\lceil \frac{2n}{j-3}\rceil=1$ it can be check that $K_{5}\subseteq G^1$, a contradiction. Hence assume that  $t=\lceil \frac{2n}{j-3}\rceil\geq 2$. So, one can so  that $K''=K_{4\times 2}\subseteq G^1[Y]$. Also as $j\geq 7$ and $|M|=n-1$, it is easy to say that there exist at least one edges of $M$ say $e=v_1v_2$, so that $v_1,v_2\notin X_i$ for each $i\in[4]$. Therefore w.l.g let $v_1\in X_5$ and $v_2\in X_6$. Now, by considering $v_1,v_2$ and $Y$, where $Y$ is the vertices of $K$ not belongs to $M$,  we have the following claim:  
	
	\begin{Claim}\label{c7}
		W.l.g. let $|N_{G^2}(v_1)\cap Y|\geq |N_{G^2}(v_2)\cap Y| $. If $|N_{G^2}(v_1)\cap Y|\geq 2$, then $|N_{G^2}(v_2)\cap Y|=0$. If $|N_{G^2}(v_1)\cap Y|=1$, then $|N_{G^2}(v_2)\cap Y|\leq 1$ and if $|N_{G^2}(v_1)\cap Y|=1$,  then $v_1$ and $v_2$ have the same neighbor in $Y$.
	\end{Claim}
	\begin{proof} By contradiction. Suppose that $\{v,v'\}\subseteq N_{G^2}(v_1)\cap Y$ and $v''\in N_{G^2}(v_2)\cap Y$, in this case, we set $ M' = (M\setminus \{v_1,v_2\}) \cup\{v_1v,v_2v''\}$. Clearly, $M'$ is a matching with $|M'| =n$, which contradicts the  $nK_2\nsubseteq G^2$. If $|N_{G^2}(v_i)\cap Y|=1$ and $v_i$ has a different neighbor then the proof is  same.
	\end{proof}
	Therefore by Claim \ref{c7} it can be check that for at least one $i\geq 4$, there exist at least one vertex of $X_i$, say $w$, so that $|N_{G^2}(w)\cap Y|\leq 1$. Hence $|N_{G^1}(w)\cap Y|\geq |Y|-1$. So, one can check that $K_5\subseteq G^1[V(K'')\cup \{w\}]$, a contradiction again.
	So, we have  $m_j(K_5, nK_2)= \lceil \frac{2n}{j-3}\rceil $ for each $j\geq 6$ and $n\geq 3$. 
	
\end{proof}
Combining Theorems  \ref{th4}, \ref{th5} and \ref{th6}, we
obtain the next theorem which characterize the exact value of the M-R-number
$m_j(K_5, nK_2)$ for $j\geq 6$ and each $n\geq 3$ as follows:

\begin{theorem}\label{t1} If $j \geq 2$, and $n\geq 1$, then:

	\[m_j(K_5, nK_2)=  \left\lbrace
	\begin{array}{ll}
		
		\infty& ~~~~~~~~~ ~if ~j=2,3,4,~~ ~~\vspace{.2 cm}\\
		\lceil\frac{2n}{j-3}\rceil& ~~~~~~ otherwise.~\vspace{.2 cm}\\
		
	\end{array}
	\right.
	\]	
	
\end{theorem}

 \subsection{General results}
\begin{theorem} \label{th7}
	Suppose that $j\geq m+2$, $m\geq 6$ and   $n\geq 3$. Given that  $m_j(K_m, (n-1)K_2)= \lceil \frac{2n-2}{j+2-m}\rceil $, then:
	\[m_j(K_m, nK_2)= \lceil \frac{2n}{j+2-m}\rceil .\]
	
\end{theorem}
\begin{proof} 
	To prove $m_j(K_m, nK_2)\leq \lceil \frac{2n}{j+2-m}\rceil$ consider  $K=K_{j\times t}$  with partite sets $X_i=\{x_1^i,x_2^i,\ldots,x_{t}^i\}$ for each $i\in[j]$ and $t=\lceil \frac{2n}{j+2-m}\rceil$. Suppose that $K$ is $2$-colorable to $(K_m,nK_2)$, that is there exist 2-edge-coloring $(G^1,G^2)$ of  $K_{j\times t}$, where $K_m\nsubseteq G^1$ and $nK_2\nsubseteq G^2$.  Thus $K_m\nsubseteq G^1[V(K')]$ where $K'=K_{j\times t'}\subseteq K$,  $t'=\lceil \frac{2n-2}{j+2-m}\rceil$, so it can be check that  $(n-1)K_2\subseteq G^2[V(K')]$.  As $t=\lceil \frac{2n}{j+2-m}\rceil$ and $|M|=n-1$ where $M$ ba a M-M in $G^2$, then one can show that the following claim is true:
	
	\begin{Claim}\label{cl2}
		There exists at least $(m-2)t+2$ vertices of $K$, say $Y$ so that the vertices of $Y$ not belongs to $M$.
		
	\end{Claim}
	\begin{proof} It can be check that $|Y|=|V(K)|-2(n-1)$, hence $|Y|=jt-2(n-1)= (m-2)t +(j+2-m)t-2(n-1)= (m-2)t +(j+2-m)\lceil \frac{2n}{j+2-m}\rceil -2(n-1)$, as $(j+2-m)\lceil \frac{2n}{j+2-m}\rceil -2(n-1)\geq 2$ we have 	$|Y|\geq (m-2)t+2$, and the proof of claim is complete.
	\end{proof}
	If $Y\cap X_i\neq \emptyset$ for at least $m$ parts of $K$, then one can check that $K_m\subseteq G^1[Y]$,  a contradiction. Hence as $|Y|=(m-2)t+2$ and $|X_i|=t$, we can assume that $Y\cap X_i\neq \emptyset$ for $m-1$ parts. W.l.g let $Y\cap X_i\neq \emptyset$ for each $i\in[m-1]$. As $n\geq 3$, and $nK_2\nsubseteq G^2$ if $t=\lceil \frac{2n}{j+2-m}\rceil=1$ it can be check that $K_{m}\subseteq G^1$, a contradiction. Hence assume that  $t=\lceil \frac{2n}{j+2-m}\rceil\geq 2$. So it can be check that $K''=K_{(m-1)\times 2}\subseteq G^1[Y]$. Also as $j\geq m+2$ and $|M|=n-1$, it is easy to say that there exist at least one edges of $M$ say $e=v_1v_2$, so that $v_1,v_2\notin X_i$ for $i\in[m-1]$. Therefore w.l.g let $v_1\in X_m$ and $v_2\in X_{m+1}$. Now, by considering $v_1,v_2$ and $Y$, where $Y$ is the vertices of $K$ not belongs to $M$,  we have the following claim:  
	
	\begin{Claim}\label{cl3}
		W.l.g. let $|N_{G^2}(v_1)\cap Y|\geq |N_{G^2}(v_2)\cap Y| $. If $|N_{G^2}(v_1)\cap Y|\geq 2$, then $|N_{G^2}(v_2)\cap Y|=0$. If $|N_{G^2}(v_1)\cap Y|=1$, then $|N_{G^2}(v_2)\cap Y|\leq 1$ and if $|N_{G^2}(v_1)\cap Y|=1$,  then $v_1$ and $v_2$ have the same neighbor in $Y$.
	\end{Claim}
	\begin{proof} By contradiction. Suppose that $\{v,v'\}\subseteq N_{G^2}(v_1)\cap Y$ and $v''\in N_{G^2}(v_2)\cap Y$, in this case, we set $ M' = (M\setminus \{v_1,v_2\}) \cup\{v_1v,v_2v''\}$. Clearly, $M'$ is a matching with $|M'| =n$, which contradicts the  $nK_2\nsubseteq G^2$. If $|N_{G^2}(v_i)\cap Y|=1$ and $v_i$ has a different neighbor then the proof is  same.
	\end{proof}
	Therefore by Claim \ref{cl3} it can be check that for at least one $i\geq m$, there exist at least one vertex of $X_i$, say $w$, so that $|N_{G^2}(w)\cap Y|\leq 1$. Hence $|N_{G^1}(w)\cap Y|\geq |Y|-1$. So, one can check that $K_m\subseteq G^1[V(K'')\cup \{w\}]$, a contradiction again.
	So, we have  $m_j(K_m, nK_2)= \lceil \frac{2n}{j+2-m}\rceil $ for each $j\geq m+2$ and $n\geq 3$. 
	
\end{proof}

\bibliographystyle{plain}
\bibliography{yas4}

\end{document}